\documentclass[11pt]{amsart}
\usepackage{amscd,amsmath,amsfonts,amssymb,graphicx,color,geometry,hyperref}
\providecommand{\U}[1]{\protect\rule{.1in}{.1in}}
\newtheorem{theorem}{Theorem}[section]

\newtheorem{corollary}[theorem]{Corollary}

\newtheorem{remark}[theorem]{Remark}

\newtheorem{lemma}[theorem]{Lemma}

\numberwithin{equation}{section}

\newtheorem*{kahane}{Kahane's Inequality}

\newcommand{\mkn}[1]{\mathcal M(#1,n)}

\newcommand{\bi}{\mathbf i}

\newcommand{\veps}{\varepsilon}
\newcommand{\onep}{\left|\frac1{\mathbf p}\right|}
\begin{document}
\title[Optimal Hardy--Littlewood type inequalities]{Optimal Hardy--Littlewood type inequalities for polynomials and multilinear operators}
\author[Albuquerque]{N. Albuquerque}
\address{Departamento de Matem\'{a}tica, \newline \indent
Universidade Federal da Para\'{\i}ba,  \newline \indent
58.051-900 - Jo\~{a}o Pessoa, Brazil.}
\email{ngalbqrq@gmail.com}

\author[Bayart]{F. Bayart}
\address{Laboratoire de Math\'ematiques, \newline\indent
Universit\'e Blaise Pascal Campus des C\'ezeaux, \newline\indent
 F-63177 Aubiere Cedex, France.}
\email{Frederic.Bayart@math.univ-bpclermont.fr}

\author[Pellegrino]{D. Pellegrino}
\address{Departamento de Matem\'{a}tica, \newline\indent
Universidade Federal da Para\'{\i}ba, \newline\indent
58.051-900 - Jo\~{a}o Pessoa, Brazil.}
\email{dmpellegrino@gmail.com and pellegrino@pq.cnpq.br}

\author[Seoane]{J. B. Seoane-Sep\'{u}lveda}
\address{Departamento de An\'{a}lisis Matem\'{a}tico,\newline\indent Facultad de Ciencias Matem\'{a}ticas, \newline\indent Plaza de Ciencias 3, \newline\indent Universidad Complutense de Madrid,\newline\indent Madrid, 28040, Spain.\newline \indent
\textsc{ and }\newline \indent  \noindent 
Instituto de Ciencias Matem\'aticas -- ICMAT \newline \indent 
calle Nicol\'as Cabrera 13--15,   \newline \indent Madrid, 28049, Spain.}
\email{jseoane@mat.ucm.es}

\subjclass[2010]{46G25, 47H60}
\keywords{Absolutely summing operators, multilinear operators, Bohnenblust-Hille inequality}
\thanks{D. Pellegrino and J.B. Seoane-Sep\'{u}lveda was supported by CNPq Grant 401735/2013-3 (PVE - Linha 2). N. Albuquerque was supported by CAPES}

\begin{abstract}
In this paper we obtain quite general and definitive forms for Hardy--Littlewood type inequalities. Moreover, when restricted to the original particular cases, our approach provides much simpler and straightforward proofs and we are able to show that in most cases the exponents involved are optimal. The technique we used is a combination of probabilistic tools and of an interpolative approach; this former technique is also employed in this paper to improve the constants for vector-valued Bohnenblust--Hille type inequalities.
\end{abstract}
\maketitle

\section{Introduction}

In 1930 Littlewood \cite{LW} has shown the following result on bilinear forms on $c_{0}\times c_{0}$, now called Littlewood's $4/3$ inequality: for any bounded bilinear form $A:c_{0}\times c_{0} \rightarrow \mathbb{C}$,
\[
\left(  \sum_{i,j=1}^{+\infty}|A(e_{i},e_{j})|^{\frac{4}{3}}\right)  ^{\frac
{3}{4}}\leq\sqrt{2}\Vert A\Vert
\]
and, moreover, the exponent $4/3$ is optimal.  From now on, $m\geq1$ is a positive
integer, $\mathbf{p}:=\left(  p_{1},\dots,p_{m}\right)  \in\left[
1,+\infty\right]  ^{m}$ and
\[
\left\vert \frac{1}{\mathbf{p}}\right\vert :=\frac{1}{p_{1}}+\dots+\frac
{1}{p_{m}}.
\]
For $1\leq p<+\infty$, let us set $X_{p}:=\ell_{p}$ and let us define
$X_{\infty}=c_{0}$. As soon as Littlewood's $4/3$ inequality appeared, it was
rapidly extended to more general frameworks. For instance:

\begin{itemize}
\item (Bohnenblust and Hille, \cite[Theorem I]{BH}, 1931 (see also \cite{dddsss})) There exists a constant $C=C(m)\geq1$ such
that
\begin{equation}
\left(  \sum\limits_{i_{1},...,i_{m}=1}^{+\infty}\left\vert A(e_{i_{^{1}}%
},...,e_{i_{m}})\right\vert ^{\frac{2m}{m+1}}\right)  ^{\frac{m+1}{2m}}\leq
C\left\Vert A\right\Vert \label{juui}%
\end{equation}
for all continuous $m$-linear forms $A:c_{0}\times\cdots\times c_{0}%
\rightarrow\mathbb{C}$ and the exponent $\frac{2m}{m+1}$ is optimal.

\item (Hardy and Littlewood, \cite{hardy.lw}, 1934 (see also \cite[page 224]{hardy})/Praciano-Pereira, \cite[Theorems A and B]{praciano}, 1981) Let $\mathbf{p}\in\left[  1,+\infty\right]  ^{m}$ with
\[
\left\vert {\frac{1}{\mathbf{p}}}\right\vert \leq\frac{1}{2},
\]
then there exists a constant $C>0$ such that, for every continuous $m$-linear
form $A:X_{p_{1}}\times\dots\times X_{p_{m}}\rightarrow\mathbb{C}$,
\begin{equation}
\left(  \sum_{i_{1},\dots,i_{m}=1}^{+\infty}|A(e_{i_{1}},\dots,e_{i_{m}%
})|^{\frac{2m}{m+1-2\left\vert {\frac{1}{\mathbf{p}}}\right\vert }}\right)
^{\frac{m+1-2\left\vert {\frac{1}{\mathbf{p}}}\right\vert }{2m}}\leq C\Vert
A\Vert.\label{juui2}%
\end{equation}

\item (Defant and Sevilla-Peris, \cite[Theorem 1]{df}, 2009) If $1\leq s\leq q\leq2,$ there exists a constant
$C>0$ such that, for every continuous $m$-linear mapping $A:c_{0}\times
\dots\times c_{0}\rightarrow\ell_{s}$, then
\[
\left(  \sum_{i_{1},\dots,i_{m}=1}^{+\infty}\Vert A(e_{i_{1}},\dots,e_{i_{m}%
})\Vert_{\ell_{q}}^{\frac{2m}{m+2\left(  \frac{1}{s}-\frac{1}{q}\right)  }}\right)
^{\frac{m+2\left(  \frac{1}{s}-\frac{1}{q}\right)  }{2m}}\leq C\Vert A\Vert.
\]

\end{itemize}

Very recently the previous results were generalized by the authors and by Dimant and Sevilla-Peris:

\begin{itemize}
\item (\cite[Corollary 1.3]{abps}, 2013) Let $1\leq s\leq q\leq2$ and $\mathbf{p}\in\left[  1,+\infty\right]  ^{m}$ such that
\begin{equation}
\frac{1}{s}-\frac{1}{q}-\left\vert{\frac{1}{\mathbf{p}}}\right\vert
\geq0.\label{assump}%
\end{equation}
Then there exists a constant $C>0$ such that, for every continuous $m$-linear mapping $A:X_{p_1} \times\dots\times X_{p_m} \rightarrow {X_{s}}$, we have
\[
\left(  \sum_{i_{1},\dots,i_{m}=1}^{+\infty}\Vert A(e_{i_{1}},\dots,e_{i_{m}%
})\Vert_{\ell_{q}}^{\frac{2m}{m+2\left(  \frac{1}{s}-\frac{1}{q}-\left\vert {\frac
{1}{\mathbf{p}}}\right\vert \right)  }}\right)  ^{\frac{m+2\left(  \frac{1}%
{s}-\frac{1}{q}-\left\vert {\frac{1}{\mathbf{p}}}\right\vert \right)  }{2m}%
}\leq C\Vert A\Vert
\]
and the exponent is optimal.
\end{itemize}

\begin{itemize}
\item (Dimant and Sevilla-Peris, \cite[Proposition 4.4]{dsp}, 2013) Let
$\mathbf{p}\in\left[  1,+\infty\right]  ^{m}$ and $s, q \in [1,+\infty]$ be such that $s\leq q$. Then there exists a constant $C>0$ such that, for every continuous $m$-linear mapping $A : X_{p_1} \times\cdots\times X_{p_m} \rightarrow X_s$, we have
\[
\left( \sum_{i_1,\dots,i_m=1}^{+\infty}
  \left\Vert A \left( e_{i_1}, \ldots, e_{i_m}\right)  \right\Vert_{\ell_q}^\rho
\right)^\frac1\rho
\leq C\Vert A\Vert,
\]
where $\rho$ is given by

\begin{itemize}
\item[(i)] If $s \leq q \leq2$, and

\begin{itemize}
\item[(a)] if $0 \leq \onep <\frac1s -\frac1q$, then $\frac{1}{\rho} = \frac12 + \frac1m \left( \frac1s - \frac1q - \onep \right)$.

\item[(b)] if $\frac1s - \frac1q \leq \onep < \frac12 + \frac1s - \frac1q$, then $\frac1\rho = \frac12 + \frac1s - \frac1q - \onep$.
\end{itemize}

\item[(ii)] If $s \leq2 \leq q$, and

\begin{itemize}
\item[(a)] if $0 \leq \onep < \frac1s - \frac12$, then $\frac1\rho = \frac12 + \frac1m \left( \frac1s - \frac12 - \onep\right)$.

\item[(b)] if $\frac1s - \frac12 \leq \onep < \frac1s$, then $\frac1\rho = \frac1s - \onep$.
\end{itemize}

\item[(iii)] If $2 \leq s \leq q$ and $0 \leq \onep < \frac1s$, then $\frac1\rho = \frac1s - \onep$.
\end{itemize}

Moreover, the exponents in the cases (ia),(iib) and (iii) are optimal. Also,
the exponent in (ib) is optimal for $\frac1s-\frac1q\leq\ \onep <\frac12$.
\end{itemize}

Our main intention, in this paper, is to improve the previous theorems in three directions.
\begin{enumerate}
\item We study in depth the remaining cases of the Dimant and Sevilla-Peris result. Surprisingly, we show that
in case (iia), the exponent given above is optimal whereas it is not optimal in case (ib) when $\onep>\frac 12$. We give a better exponent in that case and show a necessary condition on it. These two bounds coincide when $s=1$. We can summarize this into the two following statements.

\begin{theorem}
\label{nov1} Let $\mathbf{p}\in\left[  1,+\infty\right]  ^{m}$ and let
$\rho>0$. Assume moreover that either $q\geq 2$ or $q<2$ and $\onep< \frac 12$. Let

\[
\frac{1}{\lambda}
:= \frac{1}{2}+\frac{1}{s}-\frac{1}{\min\{q,2\}} - \left\vert{\frac{1}{\mathbf{p}}}\right\vert > 0.
\]

Then there exists $C>0$ such that, for every continuous $m$-linear
operator $A:X_{p_{1}}\times\dots\times X_{p_{m}}\rightarrow X_{s}$, we have%
\[
\left(  \sum_{i_{1},\dots,i_{m}=1}^{+\infty}\Vert A(e_{i_{1}},\dots,e_{i_{m}%
})\Vert_{\ell_{q}}^{\rho}\right)  ^{\frac{1}{\rho}}\leq C\Vert A\Vert
\]
if and only if
\[
\frac{m}{\rho}\leq\frac{1}{\lambda}+\frac{m-1}{\max\{\lambda,s,2\}}.
\]

\end{theorem}

%\end{enumerate}

The following table summarizes the optimal value of $\frac{1}{\rho}$ following
the respective values of $s,q,p_{1},...,p_{m}$:
\[%
\begin{array}
[c]{c|c}%
1\leq s\leq q\leq2,\ \lambda<2 & \displaystyle\frac{1}{2}+\frac{1}{ms}%
-\frac{1}{mq}-\frac{1}{m}\times\left\vert {\frac{1}{\mathbf{p}}}\right\vert
\\[3mm]\hline
& \\
1\leq s\leq q\leq2,\ \lambda\geq2,\ \onep < \frac12 & \displaystyle\frac{1}{2}+\frac{1}{s}%
-\frac{1}{q}-\left\vert {\frac{1}{\mathbf{p}}}\right\vert \\[3mm]\hline
& \\
1\leq s\leq2\leq q,\ \lambda<2 & \displaystyle\frac{m-1}{2m}+\frac{1}%
{ms}-\frac{1}{m}\times\left\vert {\frac{1}{\mathbf{p}}}\right\vert \\[3mm]\hline
& \\
1\leq s\leq2\leq q,\ \lambda\geq2 & \displaystyle\frac{1}{s}-\left\vert
{\frac{1}{\mathbf{p}}}\right\vert \\[3mm]\hline
& \\
2\leq s\leq q & \displaystyle\frac{1}{s}-\left\vert {\frac{1}{\mathbf{p}}%
}\right\vert \\
&
\end{array}
\]

We note that (\ref{juui}) and (\ref{juui2}) are recovered by Theorem \ref{nov1} just by choosing $s=1$ and $q=2$.
%\begin{remark}
%We stress that, contrary to the former cases, $s,q$ are not restricted to
%$\left[  1,2\right]  $ and no assumption like (\ref{assump}) is needed.
%\end{remark}

When $q<2$ and $\onep>\frac 12$ (observe that this automatically implies $\lambda\geq 2$), the situation
is more difficult and we get the following statement.
\begin{theorem}\label{nov2}
Let $\mathbf{p}\in\left[  1,+\infty\right]  ^{m}$, $\onep>\frac 12$, $1\leq s\leq q\leq 2$ and let
$\rho>0$. Let us consider the following property.
\begin{quote}
There exists $C>0$ such that, for every continuous $m$-linear
operator $A:X_{p_{1}}\times\dots\times X_{p_{m}}\rightarrow X_{s}$, we have%
\[
\left(  \sum_{i_{1},\dots,i_{m}=1}^{+\infty}\Vert A(e_{i_{1}},\dots,e_{i_{m}%
})\Vert_{\ell_{q}}^{\rho}\right)  ^{\frac{1}{\rho}}\leq C\Vert A\Vert.
\]
\end{quote}
\begin{itemize}
\item[(A)] The property is satisfied as soon as 
$$\frac 1\rho\leq \frac{\left(\frac 1s-\frac 1q\right)\left(\frac 1s-\onep\right)}{\frac 12-\frac 1s}.$$
\item[(B)] If the property is satisfied, then
$$\frac 1\rho\leq 2\left(1-\onep\right)\left(\frac 1s-\frac 1q\right).$$
\end{itemize}
\end{theorem}
In particular, if $s=1$, then the property is satisfied if and only if 
$$\frac 1\rho\leq 2\left(1-\onep\right)\left(1-\frac1q\right).$$

\item We give a simpler proof of the sufficient part of the Dimant and Sevilla-Peris theorem.  It turns out that it is easier to prove a more general result.

\begin{theorem}
\label{main2}Let $\mathbf{p}\in\left[  1,+\infty\right]  ^{m}$ and $1\leq
s\leq q\leq\infty$ be such that
\[
\left\vert \frac{1}{\mathbf{p}}\right\vert <\frac{1}{2}+\frac{1}{s}-\frac
{1}{\min\{q,2\}}.
\]

Let

\[
\frac{1}{\lambda}:=\frac{1}{2}+\frac{1}{s}-\frac{1}{\min\{q,2\}}-\left\vert
\frac{1}{\mathbf{p}}\right\vert.
\]

If $\lambda>0$ and $t_{1},\ldots,t_{m}\in\left[  \lambda,\max\left\{  \lambda,s,2\right\}
\right]  $ are such that
\begin{equation}
\frac{1}{t_{1}}+\cdots+\frac{1}{t_{m}}\leq\frac{1}{\lambda}+\frac{m-1}%
{\max\{\lambda,s,2\}},\label{dez11}%
\end{equation}

then there exists $C>0$ satisfying, for every continuous $m$-linear map
$A:X_{p_{1}}\times\cdots\times X_{p_{m}}\rightarrow X_{s}$,
\begin{equation}
\left(  \sum_{i_{1}=1}^{+\infty}\left(  \ldots\left(  \sum_{i_{m}=1}^{+\infty
}\left\Vert A\left(  e_{i_{1}},\ldots,e_{i_{m}}\right)  \right\Vert
_{\ell_{q}}^{t_{m}}\right)  ^{\frac{t_{m-1}}{t_{m}}}\ldots\right)  ^{\frac{t_{1}%
}{t_{2}}}\right)  ^{\frac{1}{t_{1}}}\leq C\Vert A\Vert. \label{q3}%
\end{equation}
Moreover, the exponents are optimal except eventually if $q\leq 2$ and $\onep>\frac 12$. 
\end{theorem}

\begin{remark}
\bigskip The optimality in the above theorem shall be understood in a strong
sense: when $\lambda<2,$ we prove that if $t_{1},\ldots,t_{m}\in\left[
1,+\infty\right)  $ are so that (\ref{q3}) holds then (\ref{dez11}) is
valid. When $\lambda\geq2$, note that $\lambda=\max\left\{  \lambda
,s,2\right\}  $ and we prove that if $t=t_{1}=\cdots=t_{m}$ are in $\left[
1,+\infty\right)$ and (\ref{q3}) is valid, then we have (\ref{dez11}) and, as a direct consequence, $t\geq\lambda$.
\end{remark}

\item We prove similar results for $m$-linear mappings with arbitrary codomains which assume their cotype. For a Banach space $X$, let $q_X=\inf\{q\geq 2;\ X$ has cotype $q\}.$

 The proof that $(B)$ implies $(A)$ in the theorem below appears in \cite[Proposition 4.3]{dsp}.
\begin{theorem}\label{000aaa}
Let $\mathbf{p\in}\left[  2,{+\infty}\right]  ^{m}$, let $X$ be an infinite
dimensional Banach space with cotype $q_{X}$,  $\left\vert \frac{1}{\mathbf{p}}\right\vert <\frac{1}{q_{X}}$, and let $\rho>0.$ The following
assertions are equivalent:

(A) Every  bounded $m$-linear operator $A:X_{p_{1}}\times\dots\times X_{p_{m}%
}\rightarrow X$ is such that%
\[
\sum\limits_{i_{1},...,i_{m}=1}^{{+\infty}}\left\| A(e_{i_{^{1}}}%
,...,e_{i_{m}})\right\| ^{\rho}<{+\infty}.
\]

(B) $\frac{1}{\rho}\leq\frac{1}{q_{X}}-\left\vert \frac{1}{\mathbf{p}%
}\right\vert .$
\end{theorem}

\end{enumerate}

Finally, in the last section of the paper we obtain better estimates for the constants of vector-valued Bohnenblust--Hille inequalities.

\smallskip

We conclude this introduction by noting that our theorems can be naturally stated in the context of homogeneous
polynomials. Given an $m$-homogeneous polynomial $P:X\to Y$, we denote its coefficients $(c_\alpha(P))$. 
In \cite[Lemma 5]{df}, it is shown that an inequality
$$\left(\sum_{\alpha}\|c_\alpha(P)\|^\rho\right)^{\frac 1\rho}\leq C\|P\|$$
holds for every $m$-homogeneous polynomial $P:X\to Y$ if and only if a similar inequality 
$$\left(\sum_{i_1,\dots,i_m}\|T(e_{i_1},\dots,e_{i_m})\|^\rho\right)^{\frac 1\rho}\leq C'\|T\|$$
is satisfied for every $m$-linear mapping $A:X\times\dots\times X\to Y$, where $X$ is a Banach sequence space. 

\medskip

\noindent\textsc{Notations.} For two positive integers $n,k$, we set
 \begin{eqnarray*}
\mkn{k}&:=&\big\{\bi=(i_1,\dots,i_k); \ i_1,\dots,i_k\in\{1,\dots,n\}\big\}.\\
\end{eqnarray*}
For $q\in [1,+\infty]$, $q^*$ will denote its conjugate exponent.

\section{Proof of Theorem \ref{main2} (sufficiency)\label{88}}

\bigskip Let $1\leq q\leq{+\infty}$. We recall that a Banach space $X$ has
\emph{cotype} $q$ if there is a constant $\kappa>0$ such that, no matter how
we select finitely many vectors $x_{1},...,x_{n}\in X$,
\[
\left(  \sum_{k=1}^{n}\Vert x_{k}\Vert^{q}\right)  ^{\frac{1}{q}}\leq
\kappa\left(  \int_{I}\left\Vert \sum_{k=1}^{n}r_{k}(t)x_{k}\right\Vert
^{2}dt\right)  ^{\frac{1}{2}}
\]
where $I=[0,1]$ and $r_{k}$ denotes the $k$-th Rademacher function. To cover
the case $q={+\infty}$, the left hand side should be replaced by $\max_{1\leq
k\leq n}\Vert x_{k}\Vert$. The smallest of all these constants is denoted by
$C_{q}(X)$ and named the cotype $q$ constant of $X$.
% The infimum of the
%cotypes assumed by $X$ is denoted by $q_{X}.$

An operator between Banach spaces $v:X\to Y$ is $(r,s)$-summing (with $s\leq
r\leq{+\infty}$) if there exists $C>0$ such that, for all $n\geq1$ and for
all vectors $x_{1},\dots,x_{n}\in X$,
\[
\left(  \sum_{k=1}^{n} \|vx_{k}\|^{r}\right)  ^{\frac1r}\leq C\sup_{x^{*}\in
B_{X^{*}}}\left(  \sum_{k=1}^{n}|x^{*}(x_{k})|^{s}\right)  ^{\frac1s}.
\]
The smallest constant in this inequality is denoted by $\pi_{r,s}(v)$.

We need a cotype $q$ version of \cite[Proposition 4.1]{abps}, whose proof can
be found in \cite[Proposition 3.1]{dsp}:

\begin{lemma}
\label{cot.q} Let $X$ be a Banach space, let $Y$ be a cotype $q$ space, let $r\in[1,q]$ and let $\mathbf p\in [1,+\infty]^m$
with
\[
\left\vert {\frac{1}{\mathbf{p}}}\right\vert <\frac{1}{r}-\frac{1}{q}.
\]
Define
\[
\frac{1}{\lambda}:=\frac{1}{r}-\left\vert {\frac{1}{\mathbf{p}}}\right\vert .
\]
Then, for every continuous $m$-linear map $A:X_{p_{1}}\times\cdots\times
X_{p_{m}}\rightarrow X$ and every $\left(  r,1\right)  $-summing operator
$v:X\rightarrow Y$, we have
\begin{equation}
\left(  \sum_{i_{k}}\left(  \sum_{\widehat{i_{k}}}\Vert vA(e_{i_{1}}%
,\cdots,e_{i_{m}})\Vert_{Y}^{q}\right)  ^{\lambda/q}\right)  ^{1/\lambda}%
\leq\left(  \sqrt{2}C_{q}(Y)\right)  ^{m-1}\pi_{r,1}(v)\Vert A\Vert\label{q2}%
\end{equation}
for all $k=1,...,m.$
\end{lemma}

The symbol $\sum_{\widehat{i_{k}}}$ means that we are fixing the $k$-th index
and that we are summing over all the remaining indices.

We shall deduce from this lemma the following theorem, which extends results
of \cite{abps} and \cite{dsp}:

\begin{theorem}
\label{thmain} Let $\mathbf{p}\in\left[  1,+\infty\right]  ^{m}$, $X$ be a
Banach space, $Y$ be a cotype $q$ space and $1\leq r\leq q$, with $\left\vert
\frac{1}{\mathbf{p}}\right\vert <\frac{1}{r}$. Define
\[
\frac{1}{\lambda}:=\frac{1}{r}-\left\vert \frac{1}{\mathbf{p}}\right\vert .
\]
If $t_{1},\ldots,t_{m}\in\left[  \lambda,\max\left\{  \lambda,q\right\}
\right]  $ are such that%
\[
\frac{1}{t_{1}}+\cdots+\frac{1}{t_{m}}\leq\frac{1}{\lambda}+\frac{m-1}%
{\max\{\lambda,q\}},
\]
then, for every continuous $m$-linear map $A:X_{p_{1}}\times\cdots\times
X_{p_{m}}\rightarrow X$ and every $\left(  r,1\right)  $-summing operator
$v:X\rightarrow Y$, we have {\small
\begin{equation}
\left(  \sum_{i_{1}=1}^{{+\infty}}\left(  \ldots\left(  \sum_{i_{m}=1}^{{+\infty}
}\left\Vert vA\left(  e_{i_{1}},\ldots,e_{i_{m}}\right)  \right\Vert
_{Y}^{t_{m}}\right)  ^{\frac{t_{m-1}}{t_{m}}}\ldots\right)  ^{\frac{t_{1}%
}{t_{2}}}\right)  ^{\frac{1}{t_{1}}}\leq\left(  \sqrt{2}C_{\max\left\{
\lambda,q\right\}  }(Y)\right)  ^{m-1}\pi_{r,1}(v)\left\Vert A\right\Vert .
\label{q1}%
\end{equation}
}
\end{theorem}

\begin{proof}
If $\lambda<q$, from Lemma \ref{cot.q}, we have (\ref{q1}) for
\[
\left(  t_{1},...,t_{m}\right)  =\left(  \lambda,q,...,q\right)  .
\]
Since $\lambda<q$, the mixed $\left(  \ell_{\lambda},\ell_{q}\right)  -$ norm
inequality (see \cite[Proposition 3.1]{abps}), we also have (\ref{q1}) for the
exponents
\[
\left(  t_{1},...,t_{m}\right)  =\left(  q,...,q,\lambda,q,...,q\right)
\]
with $\lambda$ in the $k$-th position, for all $k=1,...,m$. Now, using a general version of  H\"older's inequality (see \cite[Theorem 2.1]{fou}), or
interpolating
\[
\left(  \lambda,q,...,q\right)  ,\left(  q,\lambda,q,...,q\right)
,\ldots,\left(  q,...,q,\lambda\right)
\]
in the sense of \cite{abps}, we get (\ref{q1}) for all $\left(  t_{1}%
,...,t_{m}\right)  \in[\lambda,q]^{m} $ such that
\[
\frac{1}{t_{1}}+\cdots+\frac{1}{t_{m}}=\frac{1}{\lambda}+\frac{m-1}{q}%
=\frac{1}{\lambda}+\frac{m-1}{\max\{\lambda,q\}}.
\]
If $\lambda\geq q$, for any $\varepsilon>0$, let $q_{\varepsilon}%
=\lambda+\varepsilon$. So $\lambda<q_{\varepsilon}$ and this automatically implies that 
$$\onep<\frac 1r-\frac1{q_{\varepsilon}}.$$
Since $Y$ has cotype $q_{\varepsilon}>q$, we may apply Lemma
\ref{cot.q} to get
\[
\left(  \sum_{i_{1}=1}^{N}\left(  \sum_{i_{2},...,i_{m}=1}^{N}\left\Vert
vA\left(  e_{i_{1}},\ldots,e_{i_{m}}\right)  \right\Vert ^{\lambda
+\varepsilon}\right)  ^{\frac{\lambda}{^{\lambda+\varepsilon}}}\right)
^{\frac{1}{\lambda}}\leq\left(  \sqrt{2}C_{\lambda+\varepsilon}(Y)\right)
^{m-1}\pi_{r,1}(v)\Vert A\Vert
\]
for all positive integer $N$. Making $\varepsilon\rightarrow0$, we get
\[
\left(  \sum_{i_{1},...,i_{m}=1}^{N}\left\Vert vA\left(  e_{i_{1}}%
,\ldots,e_{i_{m}}\right)  \right\Vert ^{\lambda}\right)  ^{\frac{1}{\lambda}%
}\leq\left(  \sqrt{2}C_{\lambda}(Y)\right)  ^{m-1}\pi_{r,1}(v)\Vert A\Vert,
\]
for all $N$ and the proof is done.
\end{proof}

\begin{remark}
If we take $t_{1}=\dots=t_{m}$, then, upon polarization, we recover exactly
\cite[Theorem 1.2]{dsp} with a much simpler proof due to the fact that the
inequality is simpler to prove for the extremal values of $(t_{1},\dots
,t_{m})$.
\end{remark}

\medskip

We are now ready for the proof of the sufficient part of Theorem \ref{main2}.
We split the proof into three cases, and we combine Theorem \ref{thmain} with
the Bennett-Carl inequalities (\cite{ben, carl}): for $1\leq s\leq q\leq
{+\infty}$, the inclusion map $\ell_{s}\hookrightarrow\ell_{q}$ is $\left(
r,1\right)  $-summing, where the optimal $r$ is given by
\[
\frac{1}{r}:=\frac{1}{2}+\frac{1}{s}-\frac{1}{\min\{2,q\}}.
\]

\noindent\emph{(i) $s\leq q\leq2$:}  The
Bennet-Carl-inequalities ensure that the inclusion map $\ell_{s}%
\hookrightarrow\ell_{q}$ is $\left(  r,1\right)  $-summing with $\frac{1}{r} =
\frac{1}{2} + \frac{1}{s} - \frac{1}{q}$, so the results follow from Theorem
\ref{thmain}, with $t_{1},\ldots,t_{m}$ satisfying
\[
\frac{1}{t_{1}}+\cdots+\frac{1}{t_{m}}=\frac{1}{2}+\frac{1}{s}-\frac{1}%
{q}-\left\vert \frac{1}{\mathbf{p}}\right\vert +\frac{m-1}{\max\{\lambda
,2\}}.
\]

\noindent\emph{(ii) $s\leq2\leq q$:} Also by using
Bennet-Carl inequalities, $\ell_{s}\hookrightarrow\ell_{2}$ is $\left(
s,1\right)  $-summing, thus we get (\ref{q3}) applying Theorem \ref{thmain},
with $t_{1},\ldots,t_{m}$ satisfying
\[
\frac{1}{t_{1}}+\cdots+\frac{1}{t_{m}}=\frac{1}{s}-\left\vert \frac
{1}{\mathbf{p}}\right\vert +\frac{m-1}{\max\{\lambda,2\}}.
\]

\medskip

\noindent\emph{(iii) $2\leq s\leq q$:}  Since
$\ell_{s}\hookrightarrow\ell_{s}$ is $\left(  s,1\right)  $-summing, the
result follows from Theorem \ref{thmain}, with $t_{1}=\dots=t_{m}=\lambda$ and
$\lambda\geq s$, since $r=s$ and
\[
\frac{1}{\lambda}:=\frac{1}{s}-\left\vert \frac{1}{\mathbf{p}}\right\vert
\leq\frac{1}{s}.
\]

\begin{remark}

Let us set
\[
c_{qs}:=\left\{
\begin{array}
[c]{ll}%
q\,, & \mbox{ if }s\leq q\leq2\,,\\
2\,, & \mbox{ if }s\leq2\leq q\,,\\
s\,, & \mbox{ if }2\leq s\leq q.
\end{array}
\right.
\]
\medskip

With the above notations, a careful look at the proof shows that the constant
$C$ which appears in Theorem \ref{main2} is dominated by
\[
\left(  \sqrt{2}C_{\max\left\{  \lambda,s,2\right\}  }(\ell_{c_{qs}})\right)
^{m-1}\pi_{r,1}(\ell_{s}\hookrightarrow\ell_{c_{qs}})
\]

\end{remark}

\section{Proof of Theorem \ref{main2} (optimality) \label{99}}
In this section we show that the exponents in Theorem \ref{main2} are optimal except when $q\leq 2$ and $\onep>\frac 12$. More precisely, if $(t_1,\dots,t_m)\in [1,+\infty)^m$ are such that there exists $C>1$ satisfying, for any continuous multilinear map $X_{p_1}\times\dots\times X_{p_m}\to X_{s}$,
\begin{equation}
\left(  \sum_{i_{1}=1}^{{+\infty}}\left(  \ldots\left(  \sum_{i_{m}=1}^{{+\infty}
}\left\Vert A\left(  e_{i_{1}},\ldots,e_{i_{m}}\right)  \right\Vert
_{\ell_{q}}^{t_{m}}\right)  ^{\frac{t_{m-1}}{t_{m}}}\ldots\right)  ^{\frac{t_{1}%
}{t_{2}}}\right)  ^{\frac{1}{t_{1}}}\leq C \Vert A\Vert, \label{q4}%
\end{equation}
then we prove that (\ref{dez11}) holds.
When $\lambda\geq 2$, we will always assume that $t_1=\dots=t_m=t$, since $\lambda=\max\left\{  \lambda
,s,2\right\}  $ and our inequality holds true when all the exponents are equal. We split the proof into several cases. Most of the cases are a consequence of a random construction. The main tool is the following lemma,
from \cite[Lemma 6.2]{abps}.
\begin{lemma}\label{LEMPROBA}
Let $d,n\geq 1$, $q_1,\dots,q_{d+1}\in [1,+\infty]^{d+1}$ and let, for $q\geq 1$, 
$$\alpha(q)=\left\{
\begin{array}{ll}
\frac12-\frac 1q&\textrm{if }q\geq 2\\
0&\textrm{otherwise.}
\end{array}\right.$$
Then there exists a $d$-linear mapping $A:\ell_{p_1}^n\times\dots\times\ell_{p_d}^n\to \ell_{p_{d+1}}^n$ which may be written
$$A\big(x^{(1)},\dots,x^{(d)}\big)=\sum_{i_1,\dots,i_{d+1}=1}^n \pm x_{i_1}^{(1)}\cdots x_{i_d}^{(d)}e_{i_{d+1}}$$
such that
$$\|A\|\leq C_d n^{\frac 12+\alpha(p_1)+\dots+\alpha(p_d)+\alpha(p_{d+1}^*)}.$$
\end{lemma}

\subsection{Case 1: $1\leq s\leq q\leq 2$ and $\lambda<2$}
This case has already been solved in \cite[Section 6.2]{abps}, using Lemma \ref{LEMPROBA} with $d=m$ and $(q_1,\dots,q_{m+1})=(p_1,\dots,p_m,s)$.

% We apply Theorem \ref{THMGP} with $(q_1,\dots,q_l)=(p_1,\dots,p_m,s^*)$ and $\kappa=m+1$. Since $\lambda<2$, we indeed have $p_k\geq 2$ for any $k$. Let $T$ be the obtained multilinear form.  Observe that there exists an isometric correspondence between $m$-linear maps $\ell_{p_1}^n\times\cdots\times \ell_{p_m}^n\to\ell_{s}^n$ and $(m+1)$-linear maps
%$\ell_{p_1}^n\times\cdots\times \ell_{p_m}^n\times\ell_{s^*}^n\to\mathbb C$ given by
%$$\left(z\mapsto \sum_{\bi \in\mkn{m+1}}a_\bi z_{i_1}^{(1)}\cdots z_{i_{m}}^{(m)}e_{i_{m+1}}\right)\mapsto
%\left(z\mapsto \sum_{\bi \in\mkn{m+1}}a_\bi z_{i_1}^{(1)}\cdots z_{i_{m}}^{(m)}z_{i_{m+1}}^{(m+1)}\right).$$
%Hence, we get an $m$-linear map $A:\ell_{p_1}^n\times\cdots\times \ell_{p_m}^n\to\ell_{s}^n$ which may be written
%$$A(\bx)=\sum_{\bi\in\mkn{m+1}}a_\bi x_{i_1}^{(1)}\cdots x_{i_{m}}^{(m)}e_{i_{m+1}},\ |a_\bi|=1$$
%and which satisfies
%$$\|A\|\leq C n^{\frac{m}2-\left|\frac1{\mathbf p}\right|+\frac 1s}.$$
%Observe that
%$$\|A(e_{i_1},\dots,e_{i_m})\|_{\ell_{q}}=n^{1/q}.$$
%Hence,
%$$\left(  \sum_{i_{1}=1}^{{+\infty}}\left(  \ldots\left(  \sum_{i_{m}=1}^{{+\infty}
%}\left\Vert A\left(  e_{i_{1}},\ldots,e_{i_{m}}\right)  \right\Vert
%_{\ell_{q}}^{t_{m}}\right)  ^{\frac{t_{m-1}}{t_{m}}}\ldots\right)  ^{\frac{t_{1}%
%}{t_{2}}}\right)  ^{\frac{1}{t_{1}}}=n^{\frac 1q+\frac1{t_1}+\dots+\frac 1{t_m}}.$$
%We then deduce that, provided (\ref{q4}) is satisfied, it is necessary that
%$$\frac 1{t_1}+\dots+\frac 1{t_m}\leq \frac{m}2-\left|\frac 1{\mathbf p}\right|+\frac 1s-\frac 1q$$
%showing the optimality of the exponent.

\subsection{Case 2: $1\leq s\leq q\leq 2$, $\lambda\geq 2$ and $\onep\leq \frac 12$} This case has already been solved in \cite[Proposition 4.4(ib)]{dsp} using a Fourier matrix. We shall give an alternative probabilistic proof. Let $p\in [2,+\infty]$ be such that $\frac 1p=\onep$. By Lemma \ref{LEMPROBA}, there exists a linear map $T:\ell_p^n\to\ell_s^n$ which may be written $T(x)=\sum_{i,j}\veps_{i,j}x_ie_j$ with $\veps_{i,j}=\pm 1$ and such that
$$\|T\|\leq Cn^{\frac 12+\frac 12-\frac 1p+\frac 12-\frac 1{s^*}}=Cn^{\frac 12+\frac 1s-\onep}.$$
Let $A:\ell_{p_1}^n\times\cdots\times \ell_{p_m}^n\to\ell_{s}^n$ defined by 
$$A\big(x^{(1)},\dots,x^{(m)}\big):=\sum_{i,j}\veps_{i,j}x_i^{(1)}\cdots x_i^{(m)}e_j.$$
By H\"older's inequality, it is plain that $\|A\|\leq \|T\| \leq Cn^{\frac 12+\frac 1s-\onep}$. 
On the other hand, since $A(e_{i_{1}},\dots,e_{i_{m}})\neq 0$ if and only if $i_{1}=\ldots=i_{m}$,  and
$$\|A(e_i,\dots,e_i)\|_{\ell_{q}}=n^{1/q},$$
we have
$$\left(  \sum_{\bi\in\mkn{m}}\left\Vert A\left(  e_{i_{1}},\ldots,e_{i_{m}}\right)  \right\Vert
_{\ell_{q}}^{t}\right)  ^{\frac{1}{t}}=n^{\frac 1q+\frac1{t}}.$$
This clearly implies 
$$\frac 1t\leq \frac 12+\frac 1s-\frac 1q-\left|\frac1{\mathbf p}\right|.$$

\subsection{Case 3: $1\leq s\leq 2\leq q$ and $\lambda<2$} Let $p\in [0,+\infty]$ be defined by 
$$\frac 1p=\frac 1{p_m}+\frac 1{s^*}.$$
Since $\lambda<2$, it is easy to check that $p\geq 2$ and that $p_i\geq 2$ for any $i=1,\dots,m$. 
We then apply Lemma \ref{LEMPROBA} with $d=m-1$ and $(q_1,\dots,q_m)=(p_1,\dots,p_{m-1},p^*)$. We get an $(m-1)$-linear form $T:\ell_{p_1}^n\times\cdots\times \ell_{p_{m-1}}^n\to\ell_{p^*}^n$ which can be written
$$T\big(x^{(1)},\dots,x^{(m-1)}\big)=\sum_{i_1,\dots,i_m}\veps_{i_1,\dots,i_m}x_{i_1}^{(1)}\cdots x_{i_{m-1}}^{(m-1)}e_{i_m}$$
and such that 
$$\|T\|\leq C n^{\frac 12+\frac m2-\onep-\frac 1{s^*}}=Cn^{\frac{m-1}2-\onep+\frac 1s}.$$
We then define $A:\ell_{p_1}^n\times\cdots\times \ell_{p_m}^n\to\ell_{s}^n$ by
$$A\big(x^{(1)},\dots,x^{(m)}\big)=\sum_{i_1,\dots,i_m}\veps_{i_1,\dots,i_m}x^{(1)}_{i_1}\cdots x_{i_m}^{(m)} e_{i_m}.$$
Then, for any $x^{(1)},\dots,x^{(m)}\in B_{\ell_{p_1}^n}\times\dots\times B_{\ell_{p_m}^n}$, 
\begin{eqnarray*}
\|A\big(x^{(1)},\dots,x^{(m)}\big)\|&=&\sup_{y\in B_{\ell^n_{s^*}}}\left |\sum_{i_1,\dots,i_m}\veps_{i_1,\dots,i_m}x_{i_1}^{(1)}\cdots x_{i_m}^{(m)}y_{i_m}\right|\\
&\leq&\sup_{z\in B_{\ell^n_{p}}}\left |\sum_{i_1,\dots,i_m}\veps_{i_1,\dots,i_m}x_{i_1}^{(1)}\cdots x_{i_{m-1}}^{(m-1)}z_{i_m}\right|\\
&\leq&\|T\|.
\end{eqnarray*}

Moreover, given any $\bi\in\mkn{m}$, $\|A(e_{i_1},\dots,e_{i_m})\|_q=\|e_{i_m}\|_q=1$, so that
$$\left(  \sum_{i_{1}=1}^{{+\infty}}\left(  \ldots\left(  \sum_{i_{m}=1}^{{+\infty}
}\left\Vert A\left(  e_{i_{1}},\ldots,e_{i_{m}}\right)  \right\Vert
_{\ell_{q}}^{t_{m}}\right)  ^{\frac{t_{m-1}}{t_{m}}}\ldots\right)  ^{\frac{t_{1}%
}{t_{2}}}\right)  ^{\frac{1}{t_{1}}}=n^{\frac {1}{t_1}+\dots+\frac{1}{t_m}}.$$
Hence, provided (\ref{q4}) is satisfied, $(t_1,\dots,t_m)$ has to satisfy
$$\frac 1{t_1}+\dots+\frac 1{t_m}\leq \frac{m-1}2+\frac 1s-\left|\frac1{\mathbf p}\right|.$$

\subsection{Case 4 and Case 5: $1\leq s\leq2\leq q$ and $\lambda\geq2$, $2\leq s\leq q$} These cases have a
deterministic proof, as noted in \cite[Proposition 4.4 (iib), (iii)]{dsp},
considering $A:\ell_{p_{1}}^n\times\cdots\times \ell_{p_{m}}^n \rightarrow \ell^n_{s}$
given by
\[
A \big(x^{(1)},\dots,x^{(m)}\big):=\sum_{i=1}^{n}x_{i}^{(1)}\cdots x_{i}^{(m)}e_{i}.
\]

\subsection{ The proof of Theorem \ref{nov1}}

From Theorem 1.3, by choosing $t_{1}=\ldots=t_{m}$ we conclude that provided
\[
\left\vert \frac{1}{\mathbf{p}}\right\vert <\frac{1}{2}+\frac{1}{s}-\frac
{1}{\min\{q,2\}},
\]
the best exponent $\rho$ in Theorem \ref{nov1} satisfies
\[
\frac{m}{\rho}=\frac{1}{\lambda}+\frac{m-1}{\max\{\lambda,s,2{\}}}.
\]
To conclude the proof, it remains to prove that, whenever
\[
\left\vert \frac{1}{\mathbf{p}}\right\vert \geq\frac{1}{2}+\frac{1}{s}%
-\frac{1}{\min\{q,2\}},
\]
we cannot find an exponent $\rho>0$ such that (\ref{nov1}) is satisfied for all
$m$-linear operators $A:X_{p_{1}}\times\cdots\times X_{p_{m}}\rightarrow
X_{s}$. In fact, everything has already been done before: if $q\leq 2$, then we have just to follow the lines of Case 2 and if $q\geq 2$, then we may consider the $m$-linear mapping of Cases 4 and 5.

\section{The case $1\leq s\leq q\leq  2$, $\lambda\geq 2$ and $\onep>\frac12$}

\subsection{A reformulation of the Hardy-Littlewood type inequalities}

We shall improve in this section the bound given by Theorem \ref{nov1}. We shall proceed by interpolation.
To do this, we need a reformulation of the result of this theorem, as Villanueva and Perez-Garcia reformulated the Bohnenblust-Hille inequality in \cite{PV2}. The proof is a combination of  \cite[Corollary 3.20]{PV2} and Proposition 2.2 and will be omitted.

\begin{theorem}\label{THMREFORMULATION}
Let  $1\leq p_{1},...,p_{m}\leq+\infty$, $1\leq s\leq q\leq\infty$ and let
$\rho>0$. The following assertions are equivalent.
\begin{itemize}
\item[(A)] There exists $C>0$ such that, for every continuous $m$-linear mapping $A:X_{p_1}\times\cdots\times X_{p_m}\to X_s$, we have
$$\left(\sum_{i_1,\dots,i_m} \left\|A(e_{i_1},\dots,e_{i_m})\right\|_{\ell_q}^{\rho}\right)^{1/\rho}\leq C\|A\|.$$
\item[(B)] There exists $C>0$ such that, for any $n\geq 1$, for any Banach spaces $Y_1,\dots,Y_m$, for any
continuous $m$-linear mapping $S:Y_1\times\cdots\times Y_m\to X_s$, the induced operator
\begin{eqnarray*}
T:\ell_{p_1^*,w}^n(Y_1)\times\cdots\times \ell_{p_m^*,w}^n(Y_m)&\to&\ell_{\rho}^{n^m}(X_q)\\
\big(x^{(1)},\dots,x^{(m)}\big)&\mapsto&\big(S(x_{i_1}^{(1)},\dots,x_{i_m}^{(m)})\big)_{\bi\in\mkn{m}}
\end{eqnarray*}
satisfies $\|T\|\leq C\|S\|$.
\end{itemize}
\end{theorem}
We recall that, for any $p\in[1,+\infty]$ and any Banach space $Y$, 
$$\ell_{p,w}^n(Y)=\left\{(x_j)_{j=1}^n\subset Y;\ \|(x_j)\|_{w,p}:=\sup_{\varphi\in B_{Y^*}}\left(\sum_{j=1}^n |\varphi(x_j)|^p\right)^{1/p}<+\infty\right\}$$
with the appropriate modifications for $p=\infty$.

\subsection{Proof of the sufficient condition}
We now prove our better upper bound in the case $1\leq s\leq q\leq 2$, $\onep>\frac 12$ (namely we prove the first part of Theorem \ref{nov2}). Let $n\geq 1$, let $Y_1,\dots,Y_m$ be Banach spaces and let $S:Y_1\times\dots\times Y_m\to X_s$ be bounded. Let $n\geq 1$ and let $T$ be the operator induced by $S$ on 
$\mathcal Y=\ell_{p_1^*,w}^n(Y_1)\times\cdots\times \ell_{p_m^*,w}(Y_m)$, defined by
$$T\big(x^{(1)},\dots,x^{(m)}\big)=\big(S(x_{i_1}^{(1)},\dots,x_{i_m}^{(m)}\big)).$$
Then $T$ is bounded as an operator from $\mathcal Y$ into $\ell_{\infty}^{n^m}(X_s)$ (this is trivial). $T$ is also bounded as an operator from $\mathcal Y$ into $\ell_\rho^{n^m}(X_s)$ with $\frac 1\rho=\frac 1s-\onep$ 
(this is Theorem \ref{nov1} for $1\leq s\leq 2$ and $q\geq 2$). We can interpolate between these two extreme situations. Hence, let $q\in [s,2]$ and let $\theta\in [0,1]$ be such that 
$$\frac 1q=\frac{1-\theta}s+\frac \theta2\iff \theta=\frac{\frac 1s-\frac 1q}{\frac 1s-\frac 12}.$$
By \cite[Theorem 4.4.1]{berg.lofst}, $T$ is bounded as an operator from $\mathcal Y$ into $\ell_t^{n^m}(X_q)$
where
$$\frac 1t=\frac{1-\theta}{\infty}+\frac{\theta}{\rho}=\frac{\left(\frac 1s-\frac 1q\right)\left(\frac 1s-\onep\right)}{\frac 1s-\frac 12}.$$
\begin{remark}
It is easy to check that, for $1\leq s\leq q\leq 2$ and $\onep\geq\frac 12$, then the bound $\frac{\left(\frac 1s-\frac 1q\right)\left(\frac 1s-\onep\right)}{\frac 1s-\frac 12}$ is always better (namely larger) than the bound
$\frac 12+\frac 1s-\frac 1q-\onep$ obtained in Theorem \ref{nov1}.
\end{remark}

\subsection{The necessary condition}
We now prove the second part of Theorem \ref{nov2}. It also uses a probabilistic device  for linear maps when the two spaces do not need to have the same dimension. The forthcoming lemma can be found in \cite[Proposition 3.2]{ben}.
\begin{lemma}\label{LEMBENNETT}
Let $n,d\geq 1$, $1\leq p,s\leq 2$. There exists $T:\ell_p^d\to\ell_s^n$, $T(x)=\sum_{i,j}\pm x_je_i$ such that
$$\|T\|\leq C_{p,s}\max \left(d^{1/s},n^{1-\frac 1p}d^{\frac 1s-\frac 12}\right).$$
\end{lemma}
Coming back to the proof of Theorem \ref{nov2}, we first observe that we may always assume that $\onep<1$. Otherwise, we can consider
the $m$-linear map $A:X_{p_1}\times\cdots\times X_{p_m}\to X_s$ defined by
\[
A\big(x^{(1)},\dots,x^{(m)}\big)=\sum_{i\geq 1}x_i^{(1)}\dots x_i^{(m)}e_0
\]
and observe that it is bounded whereas it has infinitely many coefficients equal to 1. We then define $p\in [1,2]$ by $\frac 1p=\onep$ and we consider $T:\ell_p^d\to\ell_s^n$, $T(x)=\sum_{i,j}\veps_{i,j}x_je_i$ the map given by Lemma \ref{LEMBENNETT}.
We then define
\begin{eqnarray*}
A:\ell_{p_1}^d\times\cdots\times\ell_{p_m}^d&\to&\ell_s^n\\
\big(x^{(1)},\dots,x^{(m)}\big)&\mapsto&\sum_{i,j}\veps_{i,j}x_j^{(1)}\cdots x_j^{(m)}e_i
\end{eqnarray*}
and we observe that, by H\"older's inequality, $\|A\|\leq \|T\|$. Furthermore, 
$$\left(\sum_{i_1,\dots,i_m}\|A(e_{i_1},\dots,e_{i_m})\|_{\ell_q}^t\right)^{1/t}=n^{1/t}d^{1/q}.$$
Taking $d^{1/2}=n^{1-\frac 1p}$ (this is the optimal relation between $d$ and $n$), we get that if 
$$\left(\sum_{i_1,\dots,i_m}\|A(e_{i_1},\dots,e_{i_m})\|_{\ell_q}^t\right)^{1/t}\leq C\|A\|,$$
then it is necessary that 
$$\frac 1t\leq 2\left(1-\frac 1p\right)\left(\frac 1s-\frac 1q\right).$$
\begin{remark}
This last condition is optimal when $s=1$ or when $\onep=\frac 12$ (with, in fact, the same proof as in Case 2 above). When $1<s<2$, another necessary condition is 
$$\frac 1t\leq\frac 1s-\onep$$
(see Case 4 or Case 5 above).
\end{remark}

\section{Optimal estimates under cotype assumptions}\label{s8}

For a Banach space $X$, let $q_X:=\inf\{q\geq 2;\ X \textrm{ has cotype } q\}$. For scalar-valued multilinear operators it is easy to observe that summability
in multiple indexes behaves in a quite different way than summability in just one
index. For instance, for any bounded bilinear form $A:c_{0}\times
c_{0}\rightarrow\mathbb{C}$,
\[
\left(  \sum_{i,j=1}^{{+\infty}}|A(e_{i},e_{j})|^{\frac{4}{3}}\right)  ^{\frac
{3}{4}}\leq\sqrt{2}\Vert A\Vert
\]
and the exponent $4/3$ is optimal. But, if we sum diagonally $\left(
i=j\right)  $ the exponent $4/3$ can be reduced to $1$ since%
\[
\sum_{i=1}^{{+\infty}}|A(e_{i},e_{i})|\leq\Vert A\Vert
\]
for any bounded bilinear form $A:c_{0}\times c_{0}\rightarrow\mathbb{C}$. Now we prove Theorem \ref{000aaa} which shows  that when replacing the scalar field by infinite-dimensional
spaces the situation is quite different.

\begin{proof}

$(A)\Rightarrow(B).$ From a deep result of Maurey and Pisier (\cite{pisier} and \cite[Section 14]{djt}), $\ell_{q_{X}}$ is finitely
representable in $X$, which means that, for any $n\geq1$, one may find unit
vectors $z_{1},\dots,z_{n}\in X$ such that, for any $a_{1},\dots,a_{n}%
\in\mathbb{C}$,
\[
\sum_{i=1}^{n}\Vert a_{i}z_{i}\Vert_{X}\leq2\left(  \sum_{i=1}^{n}%
|a_{i}|^{q_{X}}\right)  ^{1/q_{X}}.
\]
We then consider the $m$-linear map $A:\ell_{p_{1}}^{n}\times\cdots
\times\ell_{p_{m}}^{n}\rightarrow X$ defined by
\[
A\left(  x^{(1)},\cdots,x^{(m)}\right)  :=\sum_{i=1}^{n}x_{i}^{(1)}\cdots
x_{i}^{(m)}z_{i}.
\]
Then, for any $(x^{(1)},\dots,x^{(m)})$ belonging to $B_{\ell_{p_{1}}^{n}%
}\times\dots\times B_{\ell_{p_{m}}^{n}}$,
\begin{align*}
\left\Vert A\left(  x^{(1)},\cdots,x^{(m)}\right)  \right\Vert  & \leq2\left(
\sum_{i=1}^{n}|x_{i}^{(1)}|^{q_{X}}\cdots|x_{i}^{(m)}|^{q_{X}}\right)
^{1/{q_{X}}}\\
& \leq2n^{\frac{1}{q_{X}}-\left\vert \frac{1}{\mathbf{p}}\right\vert }%
\end{align*}
where the last inequality follows from H\"{o}lder's inequality applied to the exponents $$\frac{p_{1}}{q_{X}},...,\frac{p_{m}}{q_{X}},\left(
1-q_{X}\left\vert \frac{1}{\mathbf{p}}\right\vert \right)  ^{-1}.$$

 On the other hand,
\[
\left(  \sum_{i=1}^{n}\Vert A\left(  e_{i},\dots,e_{i}\right)  \Vert^{\rho
}\right)  ^{1/\rho}=n^{\frac{1}{\rho}}%
\]
and we obtain $\left(  3\right)  $.

$(B)\Rightarrow(A).$ This implication is proved in \cite[Proposition 4.3]{dsp}. 
\end{proof}

If $X$ does not have cotype $q_X$, the condition remains necessary. But now we just have the following sufficient condition:
\[
\frac{m}{\rho}<\frac{1}{q_X}-\left|\frac{1}{\mathbf p}\right|.
\]

Of course, it would be nice to determine what happens in this case. A look at \cite[page 304]{djt} shows that the situation does not look simple.
 %What is really important for us is that $Id$ is $(q_X,1)$-summing and you have Banach space which does not have cotype $q_X$ and such that $Id$ is %$(q_X,1)$-summing (this is the example of Talagrand at the end of page 304). But they also say that there exist Banach spaces such that $Id$ is not %$(q_X,1)$-summing. Maybe these spaces are the $q$-convexification of the Tsirelson space. Do you have references on these spaces? The computations seem %tedious!

\bigskip As a consequence of the previous result we conclude that under certain circumstances the
concepts of absolutely summing multilinear operator and multiple summing
multilinear operator (see \cite{ir, matos, PV}) are precisely the same.

\begin{corollary}
\bigskip Let $p\in[  2,{+\infty}]  $, let $X$ be an infinite
dimensional Banach space with cotype $q_{X}<\frac{p}{m}$ and let $\rho>0$. The following
assertions are equivalent:

(A) Every bounded $m$-linear operator $A:X_{p}\times\dots\times X_{p}%
\rightarrow X$ is absolutely $\left(  \rho;p^{\ast}\right)  $-summing.

(B)  Every bounded $m$-linear operator $A:X_{p}\times\dots\times
X_{p}\rightarrow X$ is  multiple $\left(  \rho;p^{\ast}\right)  $-summing.

(C) $\frac{1}{\rho}\leq\frac{1}{q_{X}}- \frac{m}{p}.$
\end{corollary}

\bigskip We stress the equivalence between $(A)$ and $(B)$ is not true, in general. For instance, every bounded bilinear
operator $A:\ell_{2}\times\ell_{2}\rightarrow\ell_{2}$ is absolutely $\left(
1;1\right)  $-summing but this is no longer true for multiple summability.

\bigskip

\section{Constants of vector-valued Bohnenblust--Hille inequalities}

A particular case of our main result is the following vector-valued
Bohnenblust--Hille inequality (see \cite[Lemma 3]{df} and also \cite[Section 2.2]{ursula}):

\begin{theorem}
\label{61} Let $X$ be a Banach space, $Y$ a cotype $q$ space and
$v:X\rightarrow Y$ an $(r,1)$-summing operator with $1\leq r\leq q$. Then, for
all $m$-linear operators $T:c_{0}\times\cdots\times c_{0}\rightarrow X$,
\[
\left(  \sum_{i_{1},...,i_{m}=1}^{+\infty}\Vert vT\left(  e_{i_{1}}%
,\ldots,e_{i_{m}}\right)  \Vert_{Y}^{\frac{qrm}{q+(m-1)r}}\right)
^{\frac{q+\left(  m-1\right)  r}{qrm}}\leq C_{Y,m}\pi_{r,1}(v)\Vert T\Vert
\]
with $C_{Y,m}=\left(  \sqrt{2}C_{q}\left(  Y\right)  \right)  ^{m-1}.$
\end{theorem}

In this section, in Theorem \ref{ttrr}, we improve the above estimate for $C_{Y,m}$. The proof of Theorem \ref{ttrr} follows almost word by word the proof of \cite[Proposition 3.1]{bps} using \cite[Lemma 2.2]{defant33} and Kahane's inequality instead of the Khinchine inequality. We present the proof for the sake of completeness. We need the following inequality due to Kahane:

\begin{kahane}
Let $0<p,q<+\infty$. Then there is a constant $K_{p,q}>0$ for which
\[
\left(  \int_{I}\Vert\sum_{k=1}^{n}r_{k}(t)x_{k}\,dt\Vert^{q}\right)
^{\frac{1}{q}}\leq K_{p,q}\left(  \int_{I}\Vert\sum_{k=1}^{n}r_{k}%
(t)x_{k}\,dt\Vert^{p}\right)  ^{\frac{1}{p}},
\]
regardless of the choice of a Banach space $X$ and of finitely many vectors
$x_{1},...,x_{n}\in X$.
\end{kahane}

\bigskip

\begin{theorem}\label{ttrr}
\label{62}For all $m$ and all $1\leq k<m,$%
\[
C_{Y,m}\leq\left(  C_{q}(Y)K_{\frac{qrk}{q+(k-1)r},2}\right)  ^{m-k}C_{Y,k}.
\]

\end{theorem}

\begin{proof}
Let $\rho:=\frac{qrm}{q+(m-1)r}$ and to simplify notation let us write%
\[
vTe_{\mathbf{i}}=vT\left(  e_{i_{1}},\ldots,e_{i_{m}}\right)  .
\]
Let us make use of \cite[Remark 2.2]{bps} with $m\geq2,\ 1\leq k\leq m-1$ and
$s=\frac{qrk}{q+(k-1)r}$. So we have
\begin{equation}
\left(  \sum_{\mathbf{i}}\Vert vTe_{\mathbf{i}}\Vert_{Y}^{\rho}\right)
^{\frac{1}{\rho}}\leq\prod_{S\in P_{k}(m)}\left(  \sum_{\mathbf{i}_{S}}\left(
\sum_{\mathbf{i}_{\hat{S}}}\Vert vT\left(  e_{\mathbf{i}_{S}},e_{\mathbf{i}%
_{\hat{S}}}\right)  \Vert_{Y}^{q}\right)  ^{\frac{s}{q}}\right)  ^{\frac
{1}{s\binom{m}{k}}},\label{55f}%
\end{equation}
where $P_{k}(m)$ denotes the set of all subsets of $\{1,...,m\}$ with
cardinality $k.$ For sake of clarity, we shall assume that $S=\{1,\dots,k\}$.
By the multilinear cotype inequality (see \cite[Lemma 2.2]{defant33}) and the
Kahane inequality, we have {\small
\begin{align*}
&  \left(  \sum_{\mathbf{i}_{\hat{S}}}\Vert vT\left(  e_{\mathbf{i}_{S}%
},e_{\mathbf{i}_{\hat{S}}}\right)  \Vert_{Y}^{q}\right)  ^{\frac{s}{q}}\\
&  \leq\left(  C_{q}(Y)K_{s,2}\right)  ^{s(m-k)}\int_{I^{m-k}}\left\Vert
\sum_{\mathbf{i}_{\hat{S}}}r_{\mathbf{i}_{\hat{S}}}(t_{\hat{S}})vT\left(
e_{\mathbf{i}_{S}},e_{\mathbf{i}_{\hat{S}}}\right)  \right\Vert _{Y}%
^{s}\,dt_{\hat{S}}\\
&  =\left(  C_{q}(Y)K_{s,2}\right)  ^{s(m-k)}\int_{I^{m-k}}\left\Vert
vT\left(  e_{\mathbf{i}_{S}},\sum_{\mathbf{i}_{\hat{S}}}r_{\mathbf{i}_{\hat
{S}}}(t_{\hat{S}})e_{\mathbf{i}_{\hat{S}}}\right)  \right\Vert _{Y}%
^{s}\,dt_{\hat{S}}\\
&  =\left(  C_{q}(Y)K_{s,2}\right)  ^{s(m-k)}\int_{I^{m-k}}\left\Vert v\left(
T\left(  e_{i_{1}},\dots,e_{i_{k}},\sum_{i_{k+1}}r_{k+1}(t_{k+1})e_{k+1}%
,\dots,\sum_{i_{m}}r_{m}(t_{m})e_{m}\right)  \right)  \right\Vert _{Y}%
^{s}\,dt_{k+1}\dots dt_{m}\\
&
\end{align*}
} But for a fixed choice of $\left(  t_{k+1},\dots,t_{m}\right)  \in
I^{m-k}=[0,1]^{m-k}$, we know, by Theorem \ref{61}, that
\[
\sum_{i_{1},\dots,i_{k}}\left\Vert v\left(  T\left(  e_{i_{1}},\dots,e_{i_{k}%
},\sum_{i_{k+1}}r_{k+1}(t_{k+1})e_{k+1},\dots,\sum_{i_{m}}r_{m}(t_{m}%
)e_{m}\right)  \right)  \right\Vert _{Y}^{s}\newline\leq\left(  C_{Y,k}%
\pi_{r,1}(v)\Vert T\Vert\right)  ^{s}.
\]
Thus,
\begin{align}
&  \sum_{\mathbf{i}_{S}}\left(  \sum_{\mathbf{i}_{\hat{S}}}\Vert vT\left(
e_{\mathbf{i}_{S}},e_{\mathbf{i}_{\hat{S}}}\right)  \Vert_{Y}^{q}\right)
^{\frac{s}{q}}\nonumber\\
&  \leq\left(  C_{q}(Y)K_{s,2}\right)  ^{s(m-k)}\cdot\sum_{i_{1},\dots,i_{k}%
}\left\Vert v\left(  T\left(  e_{i_{1}},\dots,e_{i_{k}},\sum_{i_{k+1}}%
r_{k+1}(t_{k+1})e_{k+1},\dots,\sum_{i_{m}}r_{m}(t_{m})e_{m}\right)  \right)
\right\Vert _{Y}^{s}\nonumber\\
&  \leq\left(  \left(  C_{q}(Y)K_{s,2}\right)  ^{m-k}\pi_{r,1}(v)C_{Y,k}\Vert
T\Vert\right)  ^{s},\label{estrela}%
\end{align}
namely
\[
\left(  \sum_{\mathbf{i}_{S}}\left(  \sum_{\mathbf{i}_{\hat{S}}}\Vert
vT\left(  e_{\mathbf{i}_{S}},e_{\mathbf{i}_{\hat{S}}}\right)  \Vert_{Y}%
^{q}\right)  ^{\frac{s}{q}}\right)  ^{\frac{1}{s}}\leq\left(  C_{q}%
(Y)K_{s,2}\right)  ^{m-k}\pi_{r,1}(v)C_{Y,k}\Vert T\Vert.
\]
From (\ref{55f}) we conclude that
\[
\left(  \sum_{\mathbf{i}}\Vert vTe_{\mathbf{i}}\Vert_{Y}^{\rho}\right)
^{\frac{1}{\rho}}\leq\left(  C_{q}(Y)K_{s,2}\right)  ^{m-k}C_{Y,k}\pi
_{r,1}(v)\Vert T\Vert.
\]

\end{proof}

\bigskip When $m$ is even, the case $k=\frac{m}{2}$ recovers the constants
from \cite{ursula}.

\begin{corollary}
For all $m,$%
\[
C_{Y,m}\leq C_{q}(Y)^{m-1}%
%TCIMACRO{\tprod \limits_{k=1}^{m-1}}%
%BeginExpansion
{\textstyle\prod\limits_{k=1}^{m-1}}
%EndExpansion
K_{\frac{qrk}{q+(k-1)r},2}.
\]

\end{corollary}

\section{Other exponents}

From now on $1\leq r\leq q$ and $\left(  q_{1},\ldots,q_{m}\right)  \in\lbrack
r,q]^{m}$ so that
\[
\frac{1}{q_{1}}+\cdots+\frac{1}{q_{m}}=\frac{q+(m-1)r}{qr}=\frac{1}{r}%
+\frac{m-1}{q}%
\]
are called vector-valued Bohnenblust--Hille exponents. From Theorem
\ref{thmain} we have:

\begin{theorem}
[Multiple exponent vector-valued Bohnenblust--Hille inequality]Let $X$ be a
Banach space and $Y$ a cotype $q$ space with $1\leq r\leq q$. If $\left(
q_{1},\ldots,q_{m}\right)  \in\lbrack r,q]^{m}$ are vector-valued
Bohnenblust--Hille exponents, then there exists $C_{Y,q_{1},\dots,q_{m}}\geq1$
such that, for all $m$-linear operators $T:c_{0}\times\cdots\times
c_{0}\rightarrow X$ and every $(r,1)$-summing operator $v:X\rightarrow Y$, we
have
\begin{equation}
\left(  \sum_{i_{1}=1}^{+\infty}\dots\left(  \sum_{i_{m}=1}^{+\infty}\Vert
vTe_{\mathbf{i}}\Vert_{Y}^{q_{m}}\right)  ^{\frac{q_{m-1}}{q_{m}}}%
\dots\right)  ^{\frac{1}{q_{1}}}\leq C_{Y,q_{1},\ldots,q_{m}}\pi_{r,1}(v)\Vert
T\Vert,\label{eq}%
\end{equation}
with $C_{Y,q_{1},\ldots,q_{m}}=\left(  \sqrt{2}C_{q}\left(  Y\right)  \right)
^{m-1}.$
\end{theorem}

Our final result gives better estimates for the constants $C_{Y,q_{1}%
,\ldots,q_{m}}:$

\begin{theorem}
If $\left(  q_{1},\dots,q_{m}\right)  $ is a vector-valued Bohnenblust--Hille
exponent, then%
\[
C_{Y,q_{1}\dots,q_{m}}\leq\prod_{k=1}^{m}\left(  \left(  C_{q}(Y)K_{\frac
{kqr}{q+(k-1)r},2}\right)  ^{m-k}C_{Y,k}\right)  ^{\theta_{k}}%
\]
with
\begin{equation}
\theta_{m}=m\left(  \frac{1}{r}-\frac{1}{q}\right)  ^{-1}\left(  \frac
{1}{q_{m}}-\frac{1}{q}\right)  \label{222}%
\end{equation}
and
\begin{equation}
\theta_{k}=k\left(  \frac{1}{r}-\frac{1}{q}\right)  ^{-1}\left(  \frac
{1}{q_{k}}-\frac{1}{q_{k+1}}\right)  ,\ \text{ for }k=1,\dots,m-1. \label{333}%
\end{equation}
\bigskip
\end{theorem}

\begin{proof}
It suffices to consider $q_{i}\leq q_{j}$ whenever $i<j$. For each
$k=1,\dots,m$, define
\[
s_{k}=\frac{kqr}{q+(k-1)r}.
\]
From the proof of Theorem \ref{62} we have (\ref{eq}) for each exponent
$\left(  s_{k},\overset{k\text{ times}}{\dots},s_{k},q\dots,q\right)  .$ More
precisely, from (\ref{estrela}) we have
\[
\left(  \sum_{i_{1},\dots,i_{k}}\left(  \sum_{i_{k+1},\dots,i_{m}}\Vert
vTe_{\mathbf{i}}\Vert_{Y}^{q}\right)  ^{\frac{s_{k}}{q}}\right)  ^{\frac
{1}{s_{k}}}\leq\left(  C_{q}(Y)K_{s_{k},2}\right)  ^{m-k}C_{Y,k}\pi
_{r,1}(v)\Vert T\Vert.
\]
Consequently, for each $k=1,\dots,m$ we have
\[
C_{Y,s_{k},\overset{k\text{ times}}{\dots},s_{k},q\dots,q}\leq\left(
C_{q}(Y)K_{s_{k},2}\right)  ^{m-k}C_{Y,k}.
\]
Since every vector-valued Bohnenblust--Hille exponent $\left(  q_{1}%
,\ldots,q_{m}\right)  $ with $q_{1}\leq\cdots\leq q_{m}$ is obtained by
interpolation of $\alpha_{1},...,\alpha_{m}$ with $\alpha_{k}=\left(
s_{k},\overset{k\text{ times}}{\dots},s_{k},q\dots,q\right)  $, and
$\theta_{1},...,\theta_{m}$ as in (\ref{222}) and (\ref{333}), we conclude
that
\[
C_{Y,q_{1},\dots,q_{m}}\leq\prod_{k=1}^{m}\left(  C_{Y,s_{k},\overset{k\text{
times}}{\dots},s_{k},q,\dots,q}\right)  ^{\theta_{k}}\leq\prod_{k=1}%
^{m}\left(  \left(  C_{q}(Y)K_{s_{k},2}\right)  ^{m-k}C_{Y,k}\right)
^{\theta_{k}}%
\]

\end{proof}

A particular case of Kahane's inequality is Khintchine's inequality: if
$(\varepsilon_{i})$ is a sequence of independent Rademacher variables, then,
for any $p\in\lbrack1,2]$, there exists a constant $\mathrm{A}_{\mathbb{R},p}$
such that, for any $n\geq1$ and any $a_{1}\dots,a_{n}\in\mathbb{R}$,
\[
\left(  \sum_{i=1}^{n}|a_{i}|^{2}\right)  ^{\frac{1}{2}}\leq\mathrm{A}%
_{\mathbb{R},p}\left(  \int_{\Omega}\left\vert \sum_{i=1}^{n}a_{i}%
\varepsilon(\omega)\right\vert ^{p}d\omega\right)  ^{\frac{1}{p}}.
\]
It has a complex counterpart: for any $p\in\lbrack1,2]$, there exists a
constant $\mathrm{A}_{\mathbb{C},p}$ such that, for any $n\geq1$ and any
$a_{1}\dots,a_{n}\in\mathbb{C}$,
\[
\left(  \sum_{i=1}^{n}|a_{i}|^{2}\right)  ^{\frac{1}{2}}\leq\mathrm{A}%
_{\mathbb{C},p}\left(  \int_{\mathbb{T}^{n}}\left\vert \sum_{i=1}^{n}%
a_{i}z_{i}\right\vert ^{p}dz\right)  ^{\frac{1}{p}}.
\]
The best constants $\mathrm{A}_{\mathbb{R},p}$ and $\mathrm{A}_{\mathbb{C},p}$
are known (see \cite{haagerup} and \cite{kk}):

\begin{itemize}
\item $\mathrm{A}_{\mathbb{R},p} =
\begin{cases}
2^{\frac{1}{p}-\frac{1}{2}}, \text{ if } 0<p \leq p_{0} \approx1.847\\
\frac{1}{\sqrt{2}} \left(  \frac{\Gamma\left(  \frac{1+p}{2}\right)  }%
{\sqrt{\pi}} \right)  ^{-\frac{1}{p}}, \text{ if } p > p_{0} ;
\end{cases}
$

\item $\mathrm{A}_{\mathbb{C},p} = \Gamma\left(  \frac{1+p}{2}\right)
^{-\frac{1}{p}}$, if $1 <p \leq2$.
\end{itemize}

\bigskip

Taking $X=Y=\mathbb{K}$ and $r=1$ we obtain estimates for the constants of the
scalar-valued Bohnenblust--Hille inequality with multiple exponents:

\begin{corollary}
If $\left(  q_{1},\ldots,q_{m}\right)  \in\lbrack1,2]^{m}$ are so that%
\[
\frac{1}{q_{1}}+\cdots+\frac{1}{q_{m}}=\frac{m+1}{2},
\]
then%
\[
\left(  \sum_{i_{1}=1}^{+\infty}\dots\left(  \sum_{i_{m}=1}^{+\infty}\left\vert
T\left(  e_{i_{1}},...,e_{i_{m}}\right)  \right\vert ^{q_{m}}\right)
^{\frac{q_{m-1}}{q_{m}}}\dots\right)  ^{\frac{1}{q_{1}}}\leq C_{\mathbb{K}%
,m}^{2m\left(  \frac{1}{q_{m}}-\frac{1}{2}\right)  }\left(  \prod_{k=1}%
^{m-1}\left(  \mathrm{A}_{\mathbb{K},\frac{2k}{k+1}}^{m-k}C_{\mathbb{K}%
,k}\right)  ^{2k\left(  \frac{1}{q_{k}}-\frac{1}{q_{k+1}}\right)  }\right)
\Vert T\Vert
\]
for all $m$-linear operators $T:c_{0}\times\cdots\times c_{0}\rightarrow
\mathbb{K}.$ In particular, for complex scalars, the left hand side of the
above inequality can be replaced by
\[
\left(
%TCIMACRO{\dprod \limits_{j=1}^{m}}%
%BeginExpansion
{\displaystyle\prod\limits_{j=1}^{m}}
%EndExpansion
\Gamma\left(  2-\frac{1}{j}\right)  ^{\frac{j}{2-2j}}\right)  ^{2m\left(
\frac{1}{q_{m}}-\frac{1}{2}\right)  }\left(  \prod_{k=1}^{m-1}\left(
\Gamma\left(  \frac{3k+1}{2k+2}\right)  ^{\left(  \frac{-k-1}{2k}\right)
\left(  m-k\right)  }%
%TCIMACRO{\dprod \limits_{j=1}^{k}}%
%BeginExpansion
{\displaystyle\prod\limits_{j=1}^{k}}
%EndExpansion
\Gamma\left(  2-\frac{1}{j}\right)  ^{\frac{j}{2-2j}}\right)  ^{2k\left(
\frac{1}{q_{k}}-\frac{1}{q_{k+1}}\right)  }\right)  \Vert T\Vert.
\]

\end{corollary}

\bigskip

Acknowledgement. The authors thank the referee for important comments and suggestions.

\end{document}